\documentclass[preprint, 12pt, english]{amsart}
\usepackage{amsthm}
\usepackage{amsmath}
\usepackage{latexsym, amssymb}
\usepackage{txfonts}
\usepackage{mathtools}
\usepackage{color}
\usepackage[all]{xy}

\newtheorem{thm}{Theorem}[section] 

\newtheorem{cor}[thm]{Corollary}
\newtheorem{defn}[thm]{Definition}

\newtheorem{lem}[thm]{Lemma}

\newtheorem{ques}[thm]{Question}

\theoremstyle{definition}
\newtheorem{rem}[thm]{Remark}

\newcommand\operA[2]{{\if!#2!\operatorname{#1}\else{\operatorname{#1}_{#2}^{\phantom{I}}}\fi}} 

\newcommand\Cref[1]{{Corollary~\ref{#1}}}%
\newcommand{\Trace}[1][]{\if!#1!\operatorname{Tr}\else{\operatorname{Tr}_{#1}^{\phantom{I}}}\fi} 

\long\def\forget#1\forgotten{{}} %

\def\({\left(}
\def\){\right)}

\newcommand\LAY[3][]{{\begin{array}{c}\mbox{#2} \if#1!{}\else{+}\fi \\ \mbox{#3}\end{array}}}

\makeatletter

\def\ps@pprintTitle{%
 \let\@oddhead\@empty
 \let\@evenhead\@empty
 \def\@oddfoot{}%
 \let\@evenfoot\@oddfoot}

\newcommand{\bigperp}{%
  \mathop{\mathpalette\bigp@rp\relax}%
  \displaylimits
}

\newcommand{\bigp@rp}[2]{%
  \vcenter{
    \m@th\hbox{\scalebox{\ifx#1\displaystyle2.1\else1.5\fi}{$#1\perp$}}
  }%
}
\makeatother

\renewcommand{\geq}{\geqslant}
\renewcommand{\leq}{\leqslant}

\newif\iffurther
\furtherfalse

\begin{document}

\title{Chain Lemma, Quadratic Forms and Symbol Length}

\author{Adam Chapman}
\address{School of Computer Science, Academic College of Tel-Aviv-Yaffo, Rabenu Yeruham St., P.O.B 8401 Yaffo, 6818211, Israel}
\email{adam1chapman@yahoo.com}
\author{Ilan Levin}
\address{Department of Mathematics, Bar-Ilan University, Ramat Gan, 55200}
\email{ilan7362@gmail.com}

\begin{abstract}
We want to bound the symbol length of classes in ${_{2^{m-1}}Br}(F)$ which are represented by tensor products of 5 or 6 cyclic algebras of degree $2^m$.
The main ingredients are the chain lemma for quadratic forms, a form of a generalized Clifford invariant and Pfister's and Rost's descriptions of 12- and 14-dimensional forms in $I^3 F$.
\end{abstract}

\maketitle

\section{Introduction}

By the renowned Merkurjev-Suslin theorem (\cite{MS}), every class in ${_rBr}(F)$ is Brauer equivalent to a tensor product of cyclic algebras of degree $r$, under the assumption that $\operatorname{char}(F)$ is prime to $r$, and that $F$ contains a primitive $r$th root $\rho$. A cyclic algebra of this type takes the form
$$(\alpha,\beta)_{r,F}=F\langle i,j : i^r=\alpha, j^r=\beta, j i=\rho i j \rangle$$
for some $\alpha,\beta \in F^\times$.

Taking $r=2^m$, the symbol $(\alpha^2,\beta)_{2^m,F}$ is Brauer equivalent to $(\alpha,\beta)_{2^{m-1},F}$. Therefore, if one knows how to write a class $A$ in ${_{2^{m-1}}Br}(F)$ as a tensor product of cyclic algebras of degree $2^{m-1}$, one can easily produce a presentation of this Brauer class as a tensor product of cyclic algebras of degree $2^m$.
The opposite direction is not so simple, though:
\begin{ques}
Given a class $A$ in ${_{2^{m-1}}Br}(F)$ that can be written as a tensor product of $q$ cyclic algebras of degree $2^m$, how many cyclic algebras of degree $2^{m-1}$ are required to express it? Or equivalently, what is its symbol length in ${_{2^{m-1}}Br}(F)$?
\end{ques}

The case of $q=1$ was solved in \cite{Tignol:1983}, where it was shown that the Brauer class of a cyclic algebra of degree $2^m$ and exponent dividing $2^{m-1}$ can be written as a tensor product of at most two cyclic algebras of degree $2^{m-1}$. In \cite{ChapmanFlorenceMcKinnie:2022}, the case of $q=2$ was considered, and it was proven that the symbol length in ${_{2^{m-1}}Br}(F)$ in this case is at most 6, using the chain lemma for quaternion algebras. This number was reduced to 5 in \cite{Chapman:2022}, where upper bounds of 15 and 46 were provided for the cases of $q=3$ and 4, respectively, using Sivatski's chain lemma for biquaternion algebras (\cite{Sivatski:2012}). The goal of this paper is to answer this question for $q=5$ and $6$. The trick is to apply the chain lemma for quadratic forms and translate it carefully back to algebras, using Pfister's description of 12-dimensional (\cite{Pfister:1966}) forms and Rost's description of 14-dimensional forms in $I^3 F$.

\section{Chain Lemma for Quadratic Forms}

The chain lemma for quadratic forms is well-known in the literature, e.g., in \cite[Chapter 1, section 3]{EKM} and \cite[Chapter 1, section 5]{Lam:2005}. But for the applications, we need to count explicitly the number of steps connecting the different presentations.
Let $F$ be a field of $\operatorname{char}(F)\neq 2$ containing the square root of $-1$, and $n$ be a natural number.
Let $G=(V,E)$ be the labelled graph defined by $V=F^{\times n}$ and $(a_1,\dots,a_n)$ and $(b_1,\dots,b_n)$ are connected by a full edge $\xymatrix{ (a_1,\dots,a_n)\ar@{-}[r] & (b_1,\dots,b_n)}$ when for some $i\in \{1,\dots,n\}$, $a_k = b_k$ for all $k\neq i$ and $b_i=c^2 a_i$, and a dashed edge $\xymatrix{ (a_1,\dots,a_n)\ar@{--}[r] & (b_1,\dots,b_n)}$ when for some $1\leq i<j\leq n$, $a_k = b_k$ for all $k\neq i,j$ and $b_i=(1+\frac{a_j}{a_i})a_i$ and $b_j=(1+\frac{a_j}{a_i})a_j$.
For any $(a_1,\dots,a_n) \in V$ there is a natural ``underlying quadratic form" $\langle a_1,\dots,a_n \rangle$. The language of graph theory serves as a solid ground for discussing the chain lemma.

\begin{lem}
The pairs $(a,b)$ and $(b,a)$ are connected by at most 5 steps, of which at most three are full and two are dashed.
\end{lem}

\begin{proof}
If $a=x^2 b$ for some $x \in F^\times$, then 
$\xymatrix{ (a,b)\ar@{-}[r] & (\frac{1}{x^2}a,b)\ar@{-}[r] & (\frac{1}{x^2}a,x^2b)=(b,a)}$. Assume $a \neq x^2 b$ for any $x\in F^\times$. This means in particular that the quadratic form $\langle a,b \rangle$ is anisotropic, for we are dealing with a field containing the square-root of $-1$.
Then
\begin{eqnarray*}
\xymatrix{
(a,b) \ar@{--}[r] & (a(1+\frac{b}{a}),b(1+\frac{b}{a}))=(a+b,b+\frac{b^2}{a})\ar@{-}[d]\\ & (a+b,\frac{a^2}{b^2}(b+\frac{b^2}{a}))=(a+b,\frac{a^2}{b}+a)\ar@{--}[d]\\ 
& (\frac{a^2}{b}+2a+b, \frac{a^2}{b}+a+\frac{(\frac{a^2}{b}+a)^2}{a+b})=((\frac{a+b}{b})^2 b,(\frac{a+b}{b})^2 a) \ar@{-}[d]\\
(b,a) \ar@{-}[r]&((\frac{a+b}{b})^2 b,a) 
}
\end{eqnarray*}
\end{proof}

\begin{cor}
For any permutation $\sigma \in S_n$, the vertex $(a_1,\dots,a_n)$ is connected to $(a_{\sigma(1)},\dots,a_{\sigma(n)})$ by at most $5(n-1)$ steps, out of which at most $3(n-1)$ are full and $2(n-1)$ are dashed.
\end{cor}

\begin{proof}
This follows from the fact that every permutation can be realized by at most $n-1$ swaps.
\end{proof}

\begin{thm}\label{ChainQ}
Every two vertices with isometric underlying quadratic forms are connected by a chain of steps consisting of at most $\frac{(n+1)n}{2}+3(n-1)$ full steps and $\frac{n(n-1)}{2}+2(n-1)$ dashed steps, in total at most $n^2+5n-5$ steps.
\end{thm}

\begin{proof}
Let $( a_1,\dots,a_n )$ and $( b_1,\dots,b_n )$ be vertices of the graph such that $\langle a_1, \dots , a_n \rangle \simeq \langle b_1, \dots , b_n \rangle$. The slot $b_n$ is represented by $\langle a_1, \dots , a_n \rangle$, which means that there exist $x_1, \dots , x_n \in F$ such that $a_1x_1^2+ \dots + a_nx_n^2 = b_n$.

Let $i_1,\dots,i_\ell$ be indices in $\{1,\dots,n\}$ for which $x_{i_1},\dots,x_{i_\ell}$ are nonzero, $b_n=\sum_{j=1}^{\ell} a_{i_j} x_{i_j}^2$ and no partial sum of this series vanishes. This subsequence exists and can be explicitly found by the following procedure: First, take the sum $a_1x_1^2+ \dots + a_nx_n^2$; then, whenever there is an index $k$ for which $x_k=0$, simply erase the summand $a_k x_k^2$; and then, if there is a partial sum that vanishes, erase the entire partial sum.
Now, 
\begin{eqnarray*}
\xymatrix{
(a_{i_1},a_{i_2},\dots,a_{i_\ell}) \ar@{-}[r] & (a_{i_1}x_{i_1}^2,a_{i_2},\dots,a_{i_\ell})\ar@{-}[r] & (a_{i_1}x_{i_1}^2,a_{i_2}x_{i_2}^2,\dots,a_{i_\ell}) \ar@{-}[d] \\    & (a_{i_1}x_{i_1}^2,a_{i_2}x_{i_2}^2,\dots,a_{i_\ell}x_{i_\ell}^2) \ar@{-}[r] & \dots
}
\end{eqnarray*}
These are altogether $\ell$ full steps. Then
\begin{eqnarray*}
\xymatrix{
(a_{i_1}x_{i_1}^2,a_{i_2}x_{i_2}^2,\dots,a_{i_\ell}x_{i_\ell}^2) \ar@{--}[r] & (a_{i_1}x_{i_1}^2+a_{i_2}x_{i_2}^2, * ,\dots,a_{i_\ell}x_{i_\ell}^2) \ar@{--}[d] \\  (a_{i_1}x_{i_1}^2+a_{i_2}x_{i_2}^2+\dots+a_{i_\ell}x_{i_\ell}^2,*,\dots,*)\ar@{--}[r] & \dots
}
\end{eqnarray*}
are additional $\ell-1$ dashed steps.
This makes the $i_1$th slot equal to $b_n$, and it required at most $n$ full steps and $n-1$ dashed steps.
Similarly, it takes at most $n-1$ full steps and $n-2$ dashed steps to make another slot equal to $b_{n-1}$, and so on.
In the end, we get $(b_{\sigma(1)},\dots,b_{\sigma(n)})$ for some $\sigma \in S_n$, which is in turn connected to $(b_1,\dots,b_n)$ by at most $3(n-1)$ full steps and $2(n-1)$ dashed steps.
Altogether, $(a_1,\dots,a_n)$ is connected to $(b_1,\dots,b_n)$ by at most $\frac{(n+1)n}{2}+3(n-1)$ full steps and $\frac{n(n-1)}{2}+2(n-1)$ dashed steps. I.e., at most $n^2+5(n-1)$ steps in total.
\end{proof}

Following the same proof as above, one can produce this somewhat more technical statement:

\begin{thm}\label{ChainQ2}
If $\langle a_1,\dots,a_n \rangle \simeq \langle b_1,\dots,b_n \rangle$, then for any $k \in \{1,\dots,n-1\}$, there exist $c_{k+1},\dots,c_n \in F^\times$ such that $\langle a_1,\dots,a_n \rangle \simeq \langle b_1,\dots,b_k,c_{k+1},\dots,c_n \rangle$ and $(a_1,\dots,a_n)$ is connected to $(b_1,\dots,b_k,c_{k+1},\dots,c_n)$ by a chain consisting of at most $\frac{k(2n-k+1)}{2}+3k$ full steps, and at most $\frac{k(2n-k-1)}{2}+2k$ dashed steps. Altogether, at most $k(2n-k)+5k=2nk-k^2+5k$ steps are needed.
\end{thm}

\begin{proof}
It requires at most $n+(n-1)+\dots+(n-k+1)=\frac{k(2n-k+1)}{2}$ full steps and $(n-1)+(n-2)+\dots+(n-k)=\frac{k(2n-k-1)}{2}$ dashed steps to obtain an $n$-tuple with $b_1,\dots,b_k$ appearing somewhere in the $n$-tuple. It then requires at most $k$ swaps to move the slots to the first $k$ places. The statement thus follows.
\end{proof}

\section{Applications to Symbol Length}
Let $m$ be a natural number and $F$ a field of $\operatorname{char}(F)\neq 2$ containing primitive $2^{m+1}$th roots of unity.
Following \cite{Baek:2019}, we define:

\begin{defn}
For a given natural number $n$, set
$$ A = \{ (a_1,\dots,a_{2n}) : a_1,\dots,a_{2n} \in F^{\times}\},$$
$$ B = \{ (b_1, \dots, b_{2n+2}) : b_1,\dots,b_{2n+2}\in F^\times \wedge \prod_{i=1}^{2n+2} b_{i} \in (F^{\times})^2 \},$$

$\phi : A \rightarrow B,\text{  } \phi(a_1, \dots, a_{2n})=(b_1, \dots, b_{2n+2})$ where
 
$$b_1=1, b_{2n+2}=\prod_{j=1}^{2n} a_j,$$

and for all $i \in \{1,\dots,n\}$ we have
\begin{eqnarray*}
b_{2i} & = & a_{2i-1} \prod_{j=1}^{2i-1} b_j^{-1},\\
b_{2i+1} & = & a_{2i} \prod_{j=1}^{2i-1} b_j^{-1}.\\
\end{eqnarray*}

And set $\psi : B \rightarrow A, \text{ } \psi(b_1, \dots, b_{2n+2})=(a_1, \dots, a_{2n})$, where for all $i \in \{1,\dots,n\}$ we have
\begin{eqnarray*}
a_{2i-1} & = & b_{2i} \prod_{j=1}^{2i-1} b_j,\\
a_{2i} & = & b_{2i+1} \prod_{j=1}^{2i-1} b_j.\\
\end{eqnarray*}

\end{defn}

\begin{rem}
The composition $\psi \circ \phi$ is the identity map $\operatorname{Id}$ on $A$.
\end{rem}

\begin{proof}
Take $(a_1,\dots,a_{2n}) \in A$.
Then $(b_1,\dots,b_{2n+2})=\phi(a_1,\dots,a_{2n})$ is defined by $$b_1=1, b_{2n}=\prod_{j=1}^{2n} a_j,$$

and for all $i \in \{1,\dots,n\}$ 
\begin{eqnarray*}
b_{2i} & = & a_{2i-1} \prod_{j=1}^{2i-1} b_j^{-1},\\
b_{2i+1} & = & a_{2i} \prod_{j=1}^{2i-1} b_j^{-1}.\\
\end{eqnarray*}
Isolating the $a$'s, we obtain that for all $i \in \{1,\dots,n\}$, we have
\begin{eqnarray*}
a_{2i-1} & = & b_{2i} \prod_{j=1}^{2i-1} b_j,\\
a_{2i} & = & b_{2i+1} \prod_{j=1}^{2i-1} b_j.\\
\end{eqnarray*}
This means $\psi(b_1,\dots,b_{2n+2})=(a_1,\dots,a_{2n})$, and thus the composition $\psi \circ \phi$ is the identity map $\operatorname{Id}$ on $A$.

\end{proof}

\begin{defn}
Let $a_1,\dots,a_{2n}\in F^{\times}$. We define the following tensor product
$$C_m(a_1,\dots,a_{2n}) = (a_1,a_2)_{2^m,F} \otimes \dots \otimes (a_{2n-1},a_{2n})_{2^m,F}.$$
\end{defn}

\begin{rem}\label{DefRem}
Since for $\pi=\langle b_1,\dots,b_{2n+2} \rangle \in I^2 F$, the Clifford algebra $C\ell(\pi)$ is Brauer equivalent to $(b_1b_2,b_1b_3)_{2,F} \otimes (b_1b_2b_3b_4,b_1b_2b_3b_5)_{2,F} \otimes \dots \otimes (b_1 \dots b_{2n-1} b_{2n},b_1 \dots b_{2n-1} b_{2n+1})_{2,F}$, the algebra $C_1(\psi( b_1,\dots,b_{2n+2}))$ is in fact Brauer equivalent to the Clifford algebra of $\langle b_1,\dots,b_{2n+2} \rangle$ for any $(b_1,\dots,b_{2n+2})$ in $B$. The latter condition is the one that guarantees that  $\langle b_1,\dots,b_{2n+2} \rangle$ is in $I^2 F$. Another fact to observe is that $C_m(w)^{\otimes 2^k}$ is Brauer equivalent to $C_{m-k}(w)$ for any $k<m$ and $w\in A$.
\end{rem}

For $\varphi=(b_1,\dots,b_{2n+2}) \in B$, let $\langle \varphi \rangle$ denote the form $\langle b_1,\dots,b_{2n+2} \rangle $.

\begin{lem}\label{Step}
If two vectors $v$ and $w$ in $B$ are connected by an edge of either type, then $C_m(\psi(v))\otimes C_m(\psi(w))^{op}$ is Brauer equivalent to one symbol in ${_{2^{m-1}}Br}(F)$.
\end{lem}

\begin{proof}
Write $v=(b_1,\dots,b_{2n+2})$. Set $(a_1,\dots,a_{2n})=\psi(v)$ and $(c_1,\dots,c_{2n})=\psi(w)$.
Suppose $w$ is connected to $v$. Then $w=(b_1,\dots,b_{k-1},\alpha^2 b_k,b_{k+1},\dots,b_{2n+2})$ for some $k \in \{1,\dots,2n+2\}$ and $\alpha \in F^\times$.
If $k=1$, then $c_i=\alpha^2 a_i$ for all $i$, and so 
$C_m(\psi(w))=(c_1,c_2)_{2^m,F} \otimes \dots \otimes (c_{2n-1},c_{2n})_{2^m,F}=(\alpha^2 a_1,\alpha^2 a_2)_{2^m,F} \otimes \dots \otimes (\alpha^2 a_{2n-1},\alpha^2 a_{2n})_{2^m,F}=C_m(\psi(v)) \otimes (\alpha^2, \frac{a_2}{a_1}\cdot \frac{a_4}{a_3} \cdot \dots \cdot \frac{a_{2n}}{a_{2n-1}})_{2^m,F}$.
Therefore, $C_m(\psi(v))^{op}\otimes C_m(\psi(w))$ is Brauer equivalent to $(\alpha, \frac{a_2}{a_1}\cdot \frac{a_4}{a_3} \cdot \dots \cdot \frac{a_{2n}}{a_{2n-1}})_{2^{m-1},F}$.
Similarly, for $k\geq 2$ we also have $C_m(\psi(v))^{op}\otimes C_m(\psi(w)) \sim_{Br} (\alpha^2,*)_{2^m,F} \sim_{Br} (\alpha,*)_{2^{m-1},F}$, where $*$ is some product of the $b$'s and their inverses.

Suppose now that $v$ and $w$ are connected by a dashed edge. Then $w=(b_1,\dots,b_{k-1},b_k(1+\frac{b_\ell}{b_k}),b_{k+1},\dots,b_{\ell-1},b_\ell (1+\frac{b_\ell}{b_k}),b_{\ell+1},\dots,b_n)$ for some $1\leq k<\ell \leq 2n+2$.
There are several cases here to cover, most of which are similar to each other. We shall present explicitly two cases to demonstrate how the argument goes.

The first case we consider is when $k$ and $\ell$ are odd and $3\leq k < \ell \leq 2n+1$. The cases of $k$ or $\ell$ being even or $k=1$ follow a similar argument (with some minor adaptations), as long as $\ell \leq 2n+1$.
In this case, $c_i=a_i$ for $i\leq k-2$, $c_i=(1+\frac{b_\ell}{b_k})a_i$ for $k-1\leq i \leq \ell-2$, and $c_i=(1+\frac{b_\ell}{b_k})^2 a_i$ for $i \geq \ell-1$.
Hence, $C_m(\psi(w))=(a_1,a_2)_{2^m,F} \otimes \dots \otimes  (a_{k-2},(1+\frac{b_\ell}{b_k}) a_{k-1})_{2^m,F} \otimes ((1+\frac{b_\ell}{b_k}) a_k,(1+\frac{b_\ell}{b_k}) a_{k+1})_{2^m,F} \otimes \dots \otimes ((1+\frac{b_\ell}{b_k}) a_{\ell-2},(1+\frac{b_\ell}{b_k})^2 a_{\ell-1})_{2^m,F} \otimes \dots \otimes ((1+\frac{b_\ell}{b_k})^2 a_{2n-1},(1+\frac{b_\ell}{b_k})^2 a_{2n})_{2^m,F}$.
This is Brauer equivalent to $C_m(\psi(v)) \otimes ((1+\frac{b_\ell}{b_k}),a_{k-2}^{-1} \cdot \frac{a_{k+1}}{a_k} \cdot \ldots \cdot \frac{a_{\ell-3}}{a_{\ell-4}} a_{\ell-1})_{2^m,F} \otimes ((1+\frac{b_\ell}{b_k})^2,\beta)_{2^m,F}$ for some $\beta \in F^\times$.
Since $a_{k-2}^{-1} \cdot \frac{a_{k+1}}{a_k} \cdot \ldots \cdot \frac{a_{\ell-3}}{a_{\ell-4}} a_{\ell-1}$ is $\gamma^2 \frac{b_\ell}{b_k}$ for some $\gamma \in F^\times$, we get that $C_m(\psi(w))$ is Brauer equivalent to $C_m(\psi(v)) \otimes (1+\frac{b_\ell}{b_k},\frac{b_\ell}{b_k})_{2^m,F} \otimes (1+\frac{b_\ell}{b_k},\beta\gamma)_{2^{m-1},F}$.
Because of the assumption, that $F$ contains a primitive $2^{m+1}$th root of unity, $(1+\frac{b_\ell}{b_k},\frac{b_\ell}{b_k})_{2^m,F}=(1+\frac{b_\ell}{b_k},-\frac{b_\ell}{b_k})_{2^m,F}$ is split, and thus, $C_m(\psi(w)) \otimes C_m(\psi(v))^{op}$ is Brauer equivalent to $(1+\frac{b_\ell}{b_k},\beta\gamma)_{2^{m-1},F}$.

We now consider the case of $\ell=2n+2$ which is somewhat special. For simplicity's sake, assume $k$ is still odd and $k\geq 3$.
Then $c_i=a_i$ for $i\leq k-2$ and $c_i=(1+\frac{b_{2n+2}}{b_k})a_i$ for $k-1\leq i$.
Hence, $C_m(\psi(w))=(a_1,a_2)_{2^m,F} \otimes \dots \otimes  (a_{k-2},(1+\frac{b_{2n+2}}{b_k}) a_{k-1})_{2^m,F} \otimes ((1+\frac{b_{2n+2}}{b_k}) a_k,(1+\frac{b_{2n+2}}{b_k}) a_{k+1})_{2^m,F} \otimes \dots \otimes ((1+\frac{b_{2n+2}}{b_k}) a_{2n-1},(1+\frac{b_{2n+2}}{b_k}) a_{2n})_{2^m,F}$.
This is Brauer equivalent to $C_m(\psi(v)) \otimes (1+\frac{b_{2n+2}}{b_k},a_{k-2}^{-1} \frac{a_{k+1}}{a_k} \cdot \dots \cdot \frac{a_{2n}}{a_{2n-1}})_{2^m,F}$. Since $b_{2n+2}$ is $\alpha^2 b_1 b_2 \dots b_{2n-1}$ for some $\alpha \in F^\times$, we have $a_{k-2}^{-1} \frac{a_{k+1}}{a_k} \cdot \dots \cdot \frac{a_{2n}}{a_{2n-1}}=\beta^2 \frac{b_{2n+2}}{b_k}$ for some $\beta \in F^\times$, and thus, $C_m(\psi(w))$ is Brauer equivalent to $C_m(\psi(v)) \otimes (1+\frac{b_{2n+2}}{b_k},\frac{b_{2n+2}}{B_k})_{2^m,F} \otimes (1+\frac{b_{2n+2}}{b_k},\beta^2)_{2^m,F}$. Again, $(1+\frac{b_{2n+2}}{b_k},\frac{b_{2n+2}}{B_k})_{2^m,F}$ is split, and thus $C_m(\psi(w))\otimes C_m(\psi(v))^{op}$ is Brauer equivalent to $(1+\frac{b_{2n+2}}{b_k},\beta)_{2^{m-1},F}$.
\end{proof}

Our main theorem follows from the previous lemma.

\begin{thm}\label{CliffordThm}
Let $v = (a_1,\dots , a_{2n}) \in F^{2n}$. Suppose $\phi(v)$ is connected by $r$ steps to $w \in F^{2n+2}$ such that $C_m(\psi(w))$ is Brauer equivalent to a tensor product of $s$ cyclic algebras of degree $2^{m-1}$. Then the symbol length of $C_m(v)$ in ${_{2^{m-1}}Br}(F)$ is at most $r+s$.
\end{thm}

\begin{proof}
Let $u_0=\phi(v) , u_1 , \dots , u_{r-1}, u_r=w$ be the steps (full or dashed) connecting $\phi(v)$ and $w$.
The Brauer of $C_m(v)=C_m(\psi(u_0))$ decomposes as the tensor product of the following classes:

\begin{eqnarray*}
C_m(\psi(u_0))&\otimes & C_m(\psi(u_1))^{op}\\
C_m(\psi(u_1))&\otimes & C_m(\psi(u_2))^{op}\\
&\vdots &\\
C_m(\psi(u_{r-1}))&\otimes &C_m(\psi(u_r))^{op}\\
C_m(\psi(u_r)). & \ &
\end{eqnarray*}
The last class is a Brauer equivalent to a tensor product of $s$ cyclic algebras of degree $2^{m-1}$, and by Lemma \ref{Step}, each of the other classes is Brauer equivalent to a cyclic algebra of degree $2^{m-1}$. The statement thus follows.
\end{proof}

\begin{lem}\label{PFlem}
If $\varphi=(ae,be,bf,abe,abf,ce,cf,de,df,cde,cdf,af)$, then the algebra $C_m(\psi(\varphi))$ is Brauer equivalent to $(a^6b^8c^3d^2e^{10}f^5,\frac{f}{e})_{2^{m-1}}$.
\end{lem}

\begin{proof}
A straight forward computation shows that $C_m(\psi(\varphi))$ is
$$\begin{matrix}
 & (abe^2,abef)_{2^m,F} \otimes (a^2b^3e^3f,a^2b^3e^2f^2)_{2^m,F} \otimes (a^3b^4ce^4f^2,a^3b^4ce^3f^3)_{2^m,F} \\
 & \otimes (a^3b^4c^2de^5f^3,a^3b^4c^2de^4f^4)_{2^m,F} \otimes (a^3b^4c^3d^3e^6f^4,a^3b^4c^3d^3e^5f^5)_{2^m,F}.
\end{matrix}$$ 
Using the identity $(a,b)_{2^m,F}=(a,a^{-1}b)_{2^m,F}$ for each of the 5 terms above, and then simply multiplying them all (since they will have a common second slot, $\frac{f}{e}$), we get that it is equivalent to $(a^{12}b^{16}c^{6}d^{4}e^{20}f^{10},\frac{f}{e})_{2^m,F}$ which is actually $(a^6b^8c^3d^2e^{10}f^5,\frac{f}{e})_{2^{m-1},F}$, as desired.
\end{proof}

\begin{thm}
If $A$ is a tensor product of 5 cyclic algebras of degree $2^m$ and $A$ is of exponent dividing $2^{m-1}$, then $A$ is Brauer equivalent to a tensor product of up to $200$ cyclic algebras of degree $2^{m-1}$.
\end{thm}

\begin{proof}
Write $A=(\alpha_1,\beta_1)_{2^m,F} \otimes \dots \otimes (\alpha_5,\beta_5)_{2^m,F}$. Set $\pi=\phi(\alpha_1,\beta_1, \dots ,\alpha_5,\beta_5)$. Hence, $A=C_m(\psi(\pi))$. By Remark \ref{DefRem}, $A^{2^{m-1}} \sim_{Br} (\alpha_1,\beta_1)_{2,F} \otimes \dots \otimes (\alpha_5,\beta_5)_{2,F}$.
The form $\langle \pi \rangle$ is in $I^2 F$, but since its Clifford algebra is Brauer equivalent to $(\alpha_1,\beta_1)_{2,F} \otimes \dots \otimes (\alpha_5,\beta_5)_{2,F}$, and the latter is Brauer equivalent to $A^{2^{m-1}}$, which is split, $\langle \pi \rangle$ is actually in $I^3 F$. Since it is a 12-dimensional form in $I^3 F$, by Pfister's Theorem, (\cite{Pfister:1966}), it is isometric to 
$$\langle a,b,ab,c,d,cd \rangle \otimes \langle e,f \rangle \simeq \langle ae,be,bf,abe,abf,ce,cf,de,df,cde,cdf,af \rangle$$ for some $a,b,c,d,e,f \in F^{\times}$.
Thus,  $\pi$ is connected to $$\varphi=(ae,be,bf,abe,abf,ce,cf,de,df,cde,cdf,af)$$ by a chain consisting of at most $12^2+5\cdot (12-1)=199$ steps, by Theorem \ref{ChainQ}. By Lemma \ref{PFlem}, $C_m(\psi(\varphi))$ is Brauer equivalent to one cyclic algebra of degree $2^{m-1}$, and by 
Theorem \ref{CliffordThm} we conclude that $A=C_m(\psi(\pi))$ is Brauer equivalent to a tensor product of at most 200 cyclic algebras of degree $2^{m-1}$.
\end{proof}

Let us now use this result to find an upper bound for the symbol length of classes in ${_{2^{m-1}}Br}(F)$ which are represented by tensor products of 6 cyclic algebras of degree $2^{m}$.

\sloppy By \cite[Page 8]{Baek:2019}, every 14-dimension quadratic form $q$ in $I^3(F)$ is isometric either to 
\begin{eqnarray*}
\langle T(u),T(u)^{-1}N(u),T(v),T(v)^{-1}N(v),T(uv), T(uv)^{-1}N(uv)^{-1},T(w),T(w)N(w),\\
T(uw),T(uw)N(uw),T(vw),T(vw)N(vw),T(uvw),T(uvw)N(uvw) \rangle,
\end{eqnarray*} where $T$ is a trace form and $N$ is a norm form of some quadratic field extension $F[\sqrt{d}]/F$ and $u,v,w \in {F[\sqrt{d}]}^\times$, or to $\langle 1,1 \rangle \perp q'$ for a 12-dimensional form $q'$. Note that we assume that $-1 \in (F^{\times})^{2}$, and that the product of the initial 6 slots (in the first case) is 1, and in the second case, the product of the first two slots is 1. 

\begin{thm}
If $A$ is a tensor product of 6 cyclic algebras of degree $2^m$ and $A$ is of exponent dividing $2^{m-1}$, then $A$ is Brauer equivalent to a tensor product of up to $362$ cyclic algebras of degree $2^{m-1}$.
\end{thm}

\fussy
\begin{proof}
Write $A=(\alpha_1,\beta_1)_{2^m,F} \otimes \dots \otimes (\alpha_6,\beta_6)_{2^m,F}$. Set $\pi=\phi(\alpha_1,\beta_1, \dots ,\alpha_6,\beta_6)$. Now, $\langle \pi \rangle$ is a 14-dimensional form in $I^3(F)$. Thus, by Theorem \ref{ChainQ2}, $\pi$ is connected either to 
$$\varphi = (T(u),T(u)^{-1}N(u),T(v),T(v)^{-1}N(v),T(uv),T(uv)^{-1}N(uv)^{-1},c_{7}, \dots, c_{14})$$ by up to $6(2\cdot 14-6)+5\cdot 6=162$ steps, where $u,v \in {F[\sqrt{d}]}^\times$ and $c_{7}, \dots, c_{14} \in F^{\times}$, or to $\varphi=(1,1,c_3,\dots,c_{14})$ by up to $2(2\cdot 14-2)+5\cdot 2=62$ steps, where $c_3,\dots,c_{14} \in F^\times$. In both cases, $C_m(\psi(\varphi))$ is a product of 5 cyclic algebras of degree $2^m$ and of exponent dividing $2^{m-1}$: in the first case because $$\psi(\varphi)=(N(u),T(u)T(v),N(u)N(v),N(u)T(v)T(uv),N(u)N(v)
N(uv)^{-1},\dots)$$
$$=(N(u),T(u)T(v),N(u)N(v),N(u)T(v)T(uv),1,\dots),$$
which means $C_m(\psi(\varphi))=(N(u),T(u)T(v))_{2^m,F} \otimes (N(u)N(v),N(u)T(v)T(uv))_{2^m,F} \otimes (1,*)_{2^m,F} \otimes \dots$, i.e., out of the six cyclic algebras in the decomposition of $C_m(\psi(\varphi))$, one is split; the second case is a similar (and easier) argument. Therefore, $C_m(\psi(\varphi))$ is Brauer equivalent to a tensor product of up to 200 cyclic algebras of degree $2^{m-1}$. Consequently, $A$ is Brauer equivalent to at most $200+162=362$ cyclic algebras of degree $2^{m-1}$.
\end{proof}

\section*{Acknowledgements}

The authors are very much indebted to the referee for the thorough reading of the manuscript and the many helpful comments and suggestions.

\def\cprime{$'$}

\end{document}